\documentclass{amsart}
\usepackage{xcolor}
\usepackage{hyperref}
\hypersetup{
    colorlinks = false,
    linkbordercolor = {white},
    citebordercolor = {white}
    }
\usepackage[capitalise,noabbrev]{cleveref}
\usepackage[all]{xy}

\usepackage{amsmath,amscd}
\usepackage{amssymb}
\usepackage{mathtools}
\usepackage{stmaryrd}
\usepackage{url}
\usepackage{tikz-cd}
\usepackage{enumitem}
\usepackage{fullpage}
\usepackage{musicography}

\usepackage{physics}

\usepackage{tikz}

\usepackage{enumitem,kantlipsum}

\author{Michael K. Brown}
\author{Prashanth Sridhar}

\newcommand{\Addresses}{{
	\vskip\baselineskip
  	\footnotesize
    \noindent \textsc{Department of Mathematics, University of Alabama}
    \par\nopagebreak
    
    \medskip
    \noindent \textsc{Department of Mathematics and Statistics, Auburn University} \par\nopagebreak

    \medskip
	\noindent \textit{E-mail addresses:} \texttt{mkb0096@auburn.edu, psridhar1@ua.edu}
 }}

\numberwithin{equation}{section}

\newtheorem{lem}[equation]{Lemma}

\newtheorem{prop}[equation]{Proposition}
\newtheorem{cor}[equation]{Corollary}
\newtheorem{conj}[equation]{Conjecture}

\newtheorem{claim*}{Claim}
\newtheorem{thm}[equation]{Theorem}

\theoremstyle{definition}
\newtheorem{defn}[equation]{Definition}
\newtheorem{dfn}[equation]{Definition}
\newtheorem{example}[equation]{Example}
\newtheorem{ex}[equation]{Example}

\newtheorem{setup}[equation]{Setup}
\newtheorem{nota}[equation]{Notation}

\theoremstyle{remark}

\newtheorem{remark}[equation]{Remark}
\newtheorem{remarks}[equation]{Remarks}

\newcommand{\mfrak}[1]{\mathfrak{#1}}

\renewcommand{\k}{\Bbbk}

\renewcommand{\a}{\alpha}

\newcommand{\m}{\mfrak{m}}
\newcommand{\n}{\mfrak{n}}

\newcommand{\Hom}{\operatorname{Hom}}
\newcommand{\RHom}{\mathbf{R}\underline{\Hom}}

\newcommand{\D}{\msf{D}}
\newcommand{\kk}{\mathbf{k}}

\newcommand{\del}{\partial}
\newcommand{\Mod}{\operatorname{Mod}}

\newcommand{\id}{\operatorname{id}}

\newcommand{\gr}{\operatorname{gr}}

\def\lhom{\operatorname{\underline{Hom}}}
\def\lext{\operatorname{\underline{Ext}}}
\newcommand{\dGamma}{\mathbf{R}\Gamma}

\def\nc{\newcommand}
\nc{\on}{\operatorname}
\nc{\bideg}{\on{bideg}}
\nc{\xra}{\xrightarrow}
\def\phi{\varphi}
\nc\cB{\mathcal{B}}
\def\th{\on{th}}

\def\D{\on{D}}
\def\Db{\D^{\on{b}}}
\def\Df{\D^{\on{f}}}
\nc{\into}{\hookrightarrow}
\nc{\onto}{\twoheadrightarrow}
\nc{\LL}{\mathbf{L}}
\nc{\RR}{\mathbf{R}}
\nc{\Perf}{\on{Perf}}
\nc{\nat}{\natural}
\nc{\tors}{\on{tors}}
\nc{\Tors}{\on{Tors}}

\def\Mod{\on{Mod}}
\nc{\qgr}{\on{qgr}}
\nc{\Qgr}{\on{Qgr}}
\nc{\fQgr}{\on{Qgr}^{\on{f}}}
\nc{\colim}{\on{colim}}
\def\Z{\mathbb{Z}}
\nc{\Ext}{\on{Ext}}
\def\Dsing{\D_{\on{sg}}}
\nc{\om}{\omega}
\nc{\w}{\widetilde}
\nc{\PP}{\mathbb{P}}
\nc{\mf}{\on{mf}}
\nc{\OO}{\mathcal{O}}
\nc{\Proj}{\on{Proj}}
\nc{\Qcoh}{\on{Qcoh}}
\nc{\coh}{\on{coh}}
\nc{\Tor}{\on{Tor}}
\nc{\Modf}{\Mod^{\on{f}}}
\def\op{\on{op}}
\nc{\ce}{\coloneqq}
\def\c{\colon}
\def\k{\kk}
\nc{\Com}{\on{Com}}
\nc{\A}{\mathcal{A}}
\nc{\B}{\mathcal{B}}
\nc{\C}{\mathcal{C}}
\nc{\Sh}{\on{Sh}}
\nc{\QCoh}{\on{QCoh}}
\nc{\Coh}{\on{Coh}}
\nc{\fQCoh}{\QCoh^{\on{f}}}
\nc{\ov}{\overline}
\nc{\End}{\on{\underline{End}}}
\def\MR#1{}

\nc{\Qgrf}{\Qgr^{\on{f}}}
\nc{\uHom}{\underline{\Hom}}
\nc{\Inj}{\mathrm{Inj}}
\nc{\proj}{\mathrm{Proj}}
\nc{\spec}{\mathrm{Spec}}
\nc{\Dqgr}{\D_{\qgr}}
\nc{\co}{\colon}
\nc{\cI}{\mathcal{I}}
\nc{\tM}{\widetilde{M}}
\nc{\uExt}{\underline{\Ext}}

\usepackage{microtype}
\usepackage{comment}

\begin{document}
\title{Serre duality for dg-algebras}
\thanks{The first author was partially supported by NSF grant DMS-2302373.}
\begin{abstract}
We generalize Yekutieli-Zhang's noncommutative Serre Duality Theorem to the setting of noncommutative spaces associated to dg-algebras. As an application, we establish some finiteness properties of derived global sections over such noncommutative spaces. Along the way, we generalize Yekutieli's notion of a balanced dualizing complex to the setting of dg-algebras and establish some cases in which they exist.
\end{abstract}


\numberwithin{equation}{section}

\maketitle
\setcounter{tocdepth}{1}

\setcounter{section}{0}

\section{Introduction}

Let $\k$ be a field and $X$ a smooth and proper $\k$-variety of dimension $n$. Given a vector bundle $\mathcal{V}$ on $X$, the classical Serre Duality Theorem states that there are isomorphisms
\begin{equation}
\label{eqn:serre}
H^j(X, \mathcal{V}) \cong \Ext^{n - j}(\mathcal{V}, \om_X)^*
\end{equation}
for all $0 \le j \le n$, where $\om_X$ denotes the canonical bundle of $X$, and the superscript $*$ denotes the $\k$-dual. Serre duality has been generalized in many different directions. For instance, Grothendieck-Verdier duality extends the statement of Serre duality to proper morphisms of varieties~\cite[Theorem 3.3]{residues}, and Bondal-Kapranov's notion of a Serre functor on a triangulated category~\cite{BK} has led to the discovery of analogues of Serre duality for a host of categories arising in algebraic geometry, commutative algebra, and representation theory. We refer the reader to \cite{RV} for many examples of derived categories with Serre functors; we refer also to \cite{murfet} for 
a discussion of the Serre functor on the singularity category of a commutative Gorenstein local ring with isolated singularity (see also \cite{auslander, buchweitz}).

We focus in this paper on an extension of Serre duality to graded noncommutative algebras due to Yekutieli-Zhang~\cite{Yekutieli-Zhang}. Let $A = \bigoplus_{ i \ge 0} A_i$ be a graded Noetherian algebra such that $A_0 = \k$. Denote by $\D_{\Qgr}(A)$ the quotient of the derived category $\D(A)$ of graded right $A$-modules by its subcategory $\D^{\on{Tors}}(A)$ of complexes with torsion cohomology, i.e. complexes such that each cohomology class is annihilated by some power of the homogeneous maximal ideal. Given an object $M \in \D(A)$, we denote its image in $\D_{\Qgr}(A)$ by $\widetilde{M}$. When $A$ is commutative and generated in degree 1, the category $\D_{\Qgr}(A)$ is equivalent to the derived category of quasicoherent sheaves on the projective scheme $X = \Proj(A)$. In general, no such scheme $X$ exists, but the intuition underlying the theory of noncommutative projective schemes, pioneered by Artin-Zhang~\cite{AZ}, is that  $\D_{\Qgr}(A)$ nevertheless shares many features of the derived category of a projective scheme. One such feature we consider in this paper is a noncommutative analogue of the derived global sections functor~\cite[Section 7]{AZ}, denoted $\RR\Gamma_* \co \D_{\Qgr}(A) \to \D(A)$. Given $j \in \Z$ and $\widetilde{M}, \widetilde{N} \in \D_{\Qgr}(A)$, let $\RR^j\Gamma(\widetilde{M})$ denote the internal degree 0 part of $H^j\RR\Gamma_*(\widetilde{M})$, and let $\Ext^j(\widetilde{M}, \widetilde{N}) \ce \Hom_{\D_{\Qgr}(A)}(\widetilde{M}, \widetilde{N}[j])$. Yekutieli-Zhang proved the following noncommutative version of Serre duality:

\begin{thm}[\cite{Yekutieli-Zhang} Theorem 4.2(2)]
\label{thm:intro1}
Let $A$ be as above, and assume $A$ admits a balanced dualizing complex $R$, in the sense of \cite[Definitions 3.3 and 4.1]{YEKUTIELI1_dualizing}. Given $j \in \Z$ and an object $M$ in the bounded derived category $\Db(A)$ of graded $A$-modules, there is an isomorphism 
$$
\RR^j\Gamma(\widetilde{M}) \cong \Ext^{-j - 1}(\widetilde{M}, \widetilde{R})^*,
$$
where the superscript $*$ denotes the $\k$-dual. 
\end{thm}

For instance, it follows from~\cite[Section 4]{YEKUTIELI1_dualizing} that $A$ admits a balanced dualizing complex when it is AS-Gorenstein; see~\cite[Corollary 5.6 and Theorem 7.3]{YEKUTIELI1_dualizing} for additional examples. If $A$ is AS-Gorenstein, then $R$ is isomorphic as a right $A$-module to $A(-a)[n]$, where $n$ is the injective dimension of $A$, and $a \in \Z$ is  such that there is an isomorphism $\RHom_A(\kk, A) \simeq \kk(a)[-n]$ in $\D(A)$.   Theorem~\ref{thm:intro1} therefore implies $
\RR^j\Gamma(\widetilde{M}) \cong \Ext^{n-1 - j }(\widetilde{M}, \widetilde{A(-a)})^*$, which extends the Serre duality isomorphism~\eqref{eqn:serre} in the case of an arithmetically Gorenstein projective variety.  

\medskip
Suppose now that $A$ is a differential bigraded $\k$-algebra (Definition~\ref{def:dga})---which we abbreviate in this paper to ``dg-algebra"---and assume $A$ satisfies the conditions in Setup~\ref{setup} below. The authors observe in their previous paper \cite{brown2023orlovs} that much of Artin-Zhang's theory of noncommutative projective schemes generalizes to such dg-algebras. In particular, letting $\D(A)$ denote the derived category of differential bigraded $A$-modules (see Section~\ref{sec:prelim}), and defining $\D_{\Qgr}(A)$ as the quotient of $\D(A)$ by objects whose cohomology is annihilated by some power of the homogeneous maximal ideal of $H^0(A)$, we have a derived global sections functor
$
\RR\Gamma_* \co \D_{\Qgr}(A) \to \D(A);
$
see \cite[Section 3.1]{brown2023orlovs} or Section~\ref{sec:qgr} below. As above, we denote the image of $M \in \D(A)$ in $\D_{\Qgr}(A)$ by $\widetilde{M}$. Given $\widetilde{M}, \widetilde{N} \in \D_{\Qgr}(A)$ and $j \in \Z$, we define $\RR^j\Gamma(\widetilde{M})$ and $\Ext^j(\widetilde{M}, \widetilde{N})$ just as before. Our main result is a generalization of Yekutieli-Zhang's noncommutative Serre Duality Theorem (Theorem~\ref{thm:intro1}) to dg-algebras:

\begin{thm}
\label{thm: Serre_duality}
Let $A$ be a differential bigraded $\k$-algebra as in Setup~\ref{setup}, and assume $A$ admits a balanced dualizing dg-module $R$ (Definition~\ref{def:balanced_dualizing}). Given $j\in \mathbb{Z}$ and an object $M$ in the bounded derived category $\Db(A)$ of differential bigraded $A$-modules (see Section~\ref{sec:derived}), there is an isomorphism
    \[\RR^j\Gamma(\tM)\cong \Ext^{-j-1}(\tM,\widetilde{R})^*,\]
    where the superscript $*$ denotes the $\k$-dual.
\end{thm}

Our notion of a balanced dualizing dg-module is a generalization of \cite[Definition 4.1]{YEKUTIELI1_dualizing}. We prove in  Theorem~\ref{prop:balanced} that any dg-algebra $A$ as in Setup~\ref{setup} that is Gorenstein (Definition~\ref{def:gorenstein}) and graded commutative admits a balanced dualizing dg-module; we refer the reader to \cite[Examples 2.23---2.26]{brown2023orlovs} for several families of examples of such dg-algebras. For instance, Koszul complexes on sequences of homogeneous forms in graded Gorenstein rings admit balanced dualizing dg-modules. We also prove in Proposition~\ref{lem:fd_balanced} that if $A$ is a dg-algebra as in Setup~\ref{setup}, and $\dim_\kk H(A) < \infty$, then $A$ admits a balanced dualizing dg-module.  

The existence of balanced dualizing dg-modules for graded commutative Gorenstein dg-algebras (Theorem~\ref{prop:balanced}) is the most technically challenging result in this paper; it is an application of the theory of minimal injective resolutions of dg-modules recently developed by Minamoto~\cite{Minamoto2, minamoto}. See Remark~\ref{rem:comm} for why the graded commutativity assumption is necessary for our proof of Theorem~\ref{prop:balanced}.  We conjecture that the graded commutativity assumption may be removed from Theorem~\ref{prop:balanced}: see Conjecture~\ref{conj:balanced}.

As an application of our Serre duality result (Theorem~\ref{thm: Serre_duality}), we prove that, if $A$ is a Gorenstein dg-algebra that admits a balanced dualizing dg-module, and $M \in \Db(A)$, then the truncation $\RR\Gamma(\widetilde{M})_{\ge i}$ of the derived global sections of $\widetilde{M}$ in internal degrees at least $i$ is an object in $\Db(A)$. That is, $\RR\Gamma(\widetilde{M})_{\ge i}$ has bounded cohomology, and each $\RR^j\Gamma(\widetilde{M})_{\ge i}$ is finitely generated over $H^0(A)$: see Theorem~\ref{thm:cd} for the precise statement. As a consequence, we extend the authors' dg-algebra version of Orlov's Landau-Ginzburg/Calabi-Yau correspondence~\cite[Theorem 1.3]{brown2023orlovs} to the case where $A$ is Gorenstein and admits a balanced dualizing dg-module: see Theorem~\ref{thm:main}. We emphasize that a proof of Conjecture~\ref{conj:balanced} would allow one to remove from Theorems~\ref{thm:cd} and~\ref{thm:main} the assumption that $A$ admits a balanced dualizing dg-module.

\subsection*{Overview of the paper.} We recall in Section~\ref{sec:prelim} some background on differential bigraded algebras and their derived categories. In Section 3, we introduce a notion of local cohomology for dg-algebras $A$ as in Setup~\ref{setup}, as well as our notion of a balanced dualizing dg-module. We prove a local duality result for dg-$A$-modules (Theorem~\ref{thm: local_duality}), and we show that balanced dualizing dg-modules exist in some cases (Proposition~\ref{lem:fd_balanced} and Theorem~\ref{prop:balanced}). We prove Serre duality for dg-algebras  (Theorem~\ref{thm: Serre_duality}) in Section~\ref{sec:serreduality}. Finally, in Section~\ref{sec:app}, we discuss an application of Theorem~\ref{thm: Serre_duality} to finiteness properties of derived global sections (Theorem~\ref{thm:cd}), as well as an extension of the dg-algebra version of Orlov's Landau-Ginzburg/Calabi-Yau correspondence \cite[Theorem 1.3]{brown2023orlovs} (see Theorem~\ref{thm:main}). 

\subsection*{Acknowledgments} We thank the anonymous referee for their careful reading and helpful suggestions. 
\section{Preliminaries} 
\label{sec:prelim}
We begin with some notation and conventions. Throughout, $\k$ denotes a field. We will consider bigraded $\kk$-vector spaces $V = \bigoplus_{i, j \in \Z} V_i^j$, where the superscript typically denotes a cohomological grading, and the subscript refers to an internal grading. 
Given $v \in V^j_i$, we write $\deg(v) = i$ and $|v| = j$. 
We will denote the $m^{\th}$ shift of $V$ in internal (resp. cohomological) degree by $V(m)$ (resp. $V[m]$). That is, $V(m)_i^j = V_{i + m}^j$, and $V[m]_i^j = V_i^{j + m}$. 

\subsection{Differential bigraded algebras and modules}
\label{sec:dba}

\begin{dfn}
\label{def:dga}
A \emph{differential bigraded $\kk$-algebra} is a bigraded $\kk$-algebra $A = \bigoplus_{(i, j) \in \Z^2} A_i^j$ equipped with a degree $(0, 1)$
$\kk$-linear map $\del_A$ that squares to 0 and satisfies the Leibniz rule:
$$
\del_A(xy) = \del_A(x)y + (-1)^{|x|}x\del_A(y).
$$
\end{dfn}

In this paper, we will call differential bigraded $\k$-algebras \emph{dg-algebras}, for short. We say a dg-algebra $A$ is \emph{graded commutative} if $xy= (-1)^{|x||y|}yx$ for all homogeneous $x,y \in A$. 
We will simply call $A$ \emph{commutative} if it is commutative in the usual sense, i.e. $xy = yx$ for all $x, y \in A$.

Let $A$ be a dg-algebra. We denote the underlying bigraded $\k$-algebra of $A$ by $A^\nat$. A right (resp. left) \emph{dg-$A$-module} is a right (resp. left) bigraded $A^\nat$-module $M = \bigoplus_{i, j \in \Z} M^j_i$ equipped with a bidegree $(0, 1)$ differential that satisfies the Leibniz rule: see \cite[Definition 2.2]{brown2023orlovs} for details. All modules are assumed to be right modules unless otherwise noted. Morphisms and quasi-isomorphisms of dg-algebras and dg-modules are defined in the evident way: once again, we refer to \cite[Definitions 2.1 and 2.2]{brown2023orlovs} for details. We let $A^{\op}$ denote the \emph{opposite dg-algebra} of $A$, as defined in \cite[Definition 2.1]{brown2023orlovs}. If $A$ and $B$ are dg-algebras, a \emph{dg-$A$-$B$-bimodule} is a right dg-$A^{\op} \otimes_k B$-module. 

Given a dg-$A$-module $M$, we set $M_i \coloneqq \bigoplus_{j \in \Z} M_i^j$ for all $i \in \Z$. We observe that $A_0$ is a dg-algebra (concentrated in internal degree 0), and each $M_i$ is a dg-$A_0$-module (concentrated in internal degree $i$). Similarly, we set $M^j \coloneqq \bigoplus_{i \in \Z} M^j_i$ for all $j \in \Z$. We say a dg-algebra $A$ is \emph{connected} if 
$A_0 = A_0^0 = \k$,
and $A_i^j = 0$ when $i<0$ or $j > 0$. Notice that, if $A$ is connected, then $A_0 = \k$ is a dg-$A$-module. 

We recall the definitions of the tensor product and internal $\Hom$ for dg-$A$-modules, as described, for instance, in \cite[Section 2]{brown2023orlovs}.
Given a dg-algebra $A$, a right dg-$A$-module $M$, and a left dg-$A$-module $N$, the tensor product $M \otimes_A N$ is a dg-$\kk$-module with differential
$$
m \otimes n \mapsto d_M(m) \otimes n + (-1)^{|m|}m \otimes d_N(n).
$$
If $M$ (resp. $N$) is an $A$-$A$-bimodule, then $M \otimes_A N$ is a left (resp. right) dg-$A$-module.
If $M$ and $N$ are right dg-$A$-modules, we let $\Hom_A(M, N)$ denote the $\k$-vector space of dg-$A$-module morphisms from $M$ to $N$. We denote by $\uHom_A(M, N)$ the \emph{internal $\Hom$} from $M$ to $N$, i.e. the dg-$\k$-vector space with underlying bigraded $\k$-vector space given by the $A^\nat$-linear internal $\Hom$ space $\uHom_{A^\nat}(M, N)$ and differential
$$
\a \mapsto d_N \a - (-1)^{|\a|} \a d_M.
$$
Notice that $\Hom_A(M, N)$ is the space of bidegree $(0,0)$ cocycles in $\uHom_A(M, N)$. If $M$ (resp. $N$) is an $A$-$A$-bimodule, then $\uHom_A(M, N)$ is a right (resp. left) dg-$A$-module. 

\subsection{Derived categories of dg-algebras}
\label{sec:derived}

We let $\on{Mod}(A)$ denote the category of dg-$A$-modules, $\D(A)$ the derived category of dg-$A$-modules, and $\Db(A) \subseteq \D(A)$ the subcategory given by objects isomorphic to dg-$A$-modules that are finitely generated over $A$. We refer the reader to \cite[Definition 2.13]{brown2023orlovs} for the definitions of $K$-projective, $K$-flat, and $K$-injective resolutions of dg-modules. Given a pair $M$ and $N$ of dg-modules over a dg-algebra $A$, $K$-flat resolutions $F^M$ and $F^N$ of $M$ and $N$, a $K$-projective resolution $P$ of $M$, and a $K$-injective resolution $I$ of $N$, the derived tensor product of $M$ and $N$ and derived Hom from $M$ to $N$ may be modeled as follows:
$$
M \otimes_A^{\LL} N \cong F^M \otimes_A N \cong M \otimes_A F^N, \quad  \RHom_A(M, N) \cong \underline{\Hom}_A(P, N) \cong \underline{\Hom}_A(M, I);
$$
where the isomorphisms are in the derived category $\D(\k)$. In particular, we write:
$$
\underline{\Ext}^i_A(M, N) \coloneqq H^i\RHom_A(M, N) \quad \text{and} \quad \Ext^i_A(M, N) \ce \underline{\Ext}^i_A(M, N)_0.
$$
\begin{remark}
\label{homotopy}
Let $K(A)$ denote the homotopy category of dg-$A$-modules. For any dg-$A$-module $M$ and $K$-injective dg-$A$-module $I$, the natural map
$
\Hom_{K(A)}(M,I)\to \Hom_{\D(A)}(M,I)
$
is an isomorphism \cite[Theorem 10.1.13]{Yekutieli2020}. Similarly, if $N$ is a dg-$A$-module, and $P$ is a $K$-projective dg-$A$-module, then the natural map
$
\Hom_{K(A)}(P,N)\to \Hom_{\D(A)}(P,N)
$
is an isomorphism \cite[Theorem 10.2.9]{Yekutieli2020}.
\end{remark}

\begin{prop}
\label{prop:injbimodule}
If $M$ and $N$ are dg-$A$-$A$-bimodules, then $\RHom_A(M, N)$ and $\RHom_{A^{\op}}(M, N)$ may be equipped with the structure of dg-$A$-$A$-bimodules.
\end{prop}

To prove Proposition~\ref{prop:injbimodule}, we will need the following technical result. 

\begin{lem}\label{homotopy_flat}
Let $A$ be a dg-algebra. 
\begin{enumerate}
\item Let $f : A \to B$ be a dg-algebra morphism such that $B$ is $K$-flat as a left dg-$A$-module. If $I$ is a $K$-injective right dg-$B$-module, then it is also a $K$-injective right dg-$A$-module. 
    \item Let $N$ be a dg-$A$-$A$-bimodule, and let $I$ be a $K$-injective resolution of $N$ as an $A^{\op} \otimes_\k A$-module. The resolution $I$ is $K$-injective as both a right dg-$A$-module and a right dg-$A^{\op}$-module. Thus, $I$ is a $K$-injective resolution of $N$ as both a right dg-$A$-module and a right dg-$A^{\op}$-module. 
\end{enumerate}
\end{lem}

\begin{proof}
  To prove (1), let $M$ be an exact right dg-$A$-module. Since $B$ is a $K$-flat left dg-$A$-module, the right dg-$B$-module $M \otimes_A B$ is exact. It follows that $\lhom_B(M \otimes_A B, I)$ is exact. By adjunction, we have
  $
\lhom_B(M \otimes_A B, I) \cong \lhom_A(M, \lhom_B(B, I)) \cong \lhom_A(M, I). 
$
This proves (1). 
As for (2): since $A^{\op} \otimes_\k A$ is $K$-flat as both a left dg-$A$-module and a left dg-$A^{\op}$-module, Lemma~\ref{homotopy_flat} implies that $I$ is $K$-injective as both a right dg-$A$-module and a right dg-$A^{\op}$-module. 
\end{proof}

\begin{proof}[Proof of Proposition~\ref{prop:injbimodule}]
Choose a $K$-injective resolution $I$ of $N$ as a dg-$A^{\op} \otimes_\k A$-module. By Lemma~\ref{homotopy_flat}(2), we see that $\RHom_A(M, N)$ can be modeled as $\underline{\Hom}_A(M, I)$, and $\RHom_{A^{\op}}(M, N)$ can be modeled as $\underline{\Hom}_{A^{\op}}(M, I)$. It follows that both $\RHom_A(M, N)$ and $\RHom_{A^{\op}}(M, N)$ may be equipped with dg-$A$-$A$-bimodule structures. 
\end{proof}

\subsection{Gorenstein dg-algebras}

For the rest of the paper, we will work under the following setup, which is identical to \cite[Setup 2.8]{brown2023orlovs}:

\begin{setup}
\label{setup}
Let $A$ be a connected dg-algebra such that $H^0(A)$ is Noetherian, and the total cohomology algebra $H(A)$ is finitely generated as an $H^0(A)$-module. 
\end{setup}

\begin{remarks}
\label{rem:easy}
Let $A$ be a dg-algebra as in Setup~\ref{setup}.
\begin{enumerate}
\item Let $\Df(A)$ denote the subcategory of $\D(A)$ given by dg-$A$-modules $M$ such that the total cohomology $H(M)$ is finitely generated over the total cohomology $H(A)$. By \cite[Proposition 2.17]{brown2023orlovs}, the canonical map $\Db(A) \to \Df(A)$ is an equivalence.
\item If $M \in \Db(A)$, then it follows from \cite[Proposition 2.16]{brown2023orlovs} that $\dim_\kk H(M)_\ell < \infty$ for all $\ell \in \Z$. 
\end{enumerate}
\end{remarks}

\begin{defn}\label{def:gorenstein}
Let $A$ be as in Setup~\ref{setup}. We say $A$ is \emph{Gorenstein} if:
\begin{enumerate}
\item The functor $\RHom_A( - , A)$ maps $\Db(A)$ to $\Db(A^{\op})$, and  $\RHom_{A^{\op}}( - , A)$ maps $\Db(A^{\op})$ to $\Db(A)$.
\item Given $M \in \Db(A)$ and $N \in \Db(A^{\op})$, the canonical maps
$$
M \to \RHom_{A^{\op}}(\RHom_A(M, A), A) \quad \text{and} \quad N \to \RHom_{A}(\RHom_{A^{\op}}(N, A), A)
$$
are isomorphisms in $\D(A)$ and $\D(A^{\op})$ respectively.
\item There is an isomorphism $\RHom_A(\k, A) \cong \k(a)[-n]$ in $\D(A)$ for some $a, n \in \Z$. 
\end{enumerate}
When $A$ is Gorenstein, the integer $a$ in (3) is called the \emph{Gorenstein parameter} of $A$.
\end{defn}

There are other definitions of Gorenstein dg-algebras in the literature: see \cite[Remark 2.19]{brown2023orlovs} for a detailed discussion. We refer to \cite[Examples 2.23---2.27]{brown2023orlovs} for several examples of Gorenstein dg-algebras.

\subsection{Minamoto's dg version of minimal injective resolutions}
\label{sec:minamoto}
Let $A$ be as in Setup~\ref{setup}. We now discuss (a bigraded variant of) a notion of minimal injective resolutions for dg-modules recently developed by Minamoto~\cite{Minamoto2, minamoto}. We first observe that, since $A$ is connected, the inclusion $A^0 \to A$ is a morphism of dg-algebras; in particular, $A$ is a dg-$A^0$-module. Moreover, the complex $\lhom_{A^0}(A,M)$ is naturally a dg-$A$-module for any dg-$A^0$-module $M$.

\begin{nota}(\cite{minamoto} Definition 3.9)
   Let $\Inj(A^0)$ denote the category of graded injective $A^0$-modules (not dg-modules), and let $\cI$ be the full subcategory of $\D(A)$ given by objects isomorphic to $\lhom_{A^0}(A,K)$ for some $K\in \Inj(A^0)$. 
\end{nota}

\begin{remark}
\label{rem:semiinj}
Any object in $\cI$ is both $K$-injective and injective as an $A^{\nat}$-module~\cite[Corollary 3.13]{minamoto}.
\end{remark}
The following is a bigraded analogue of \cite[Theorem 3.10]{minamoto}; its proof is essentially the same. 

\begin{thm}[cf. \cite{minamoto} Theorem 3.10]\label{thm:minamoto}
Let $A$ be as in Setup~\ref{setup}.
\begin{enumerate}
    \item If $M\in \D(A)$, and $I\in \mathcal{I}$, then 
    $\Hom_{\D(A)}(M,I)\cong \Hom_{H^0(A)}(H^{0}(M),H^0(I))$.

    \item The $0$-th cohomology functor $H^0:\D(A)\rightarrow \Mod(H^0(A))$ induces an equivalence 
    \[H^0: \mathcal{I} \xra{\simeq} \Inj(H^0(A)).\]
\end{enumerate}
\end{thm}

\begin{proof}
Let $M\in \D(A)$ and $I \in \cI$. Choose $K\in \Inj(A^0)$ such that $I \cong \lhom_{A^0}(A,K)$. We note that:
\begin{equation}
\label{eqn:injective}
\lhom_{A^0}(H^0(N),K) \cong H^0(\lhom_{A^0}(N,K)) \quad \text{for any dg-$A$-module $N$ and any injective $A^0$-module $K$}.
\end{equation} 
We have:
\begin{align*}
\Hom_{\D(A)}(M, I)             \cong \Hom_{K(A)}(M,I) 
             = H^0(\lhom_A(M,I))_0 
                             & \cong H^0(\lhom_{A^0}(M,K))_0 \\
                                       & \cong  \Hom_{A^0}(H^0(M),K) \\
                                       & \cong  \Hom_{A^0}(H^0(M),\lhom_{A^0}(H^0(A),K)),
        \end{align*}
where the first isomorphism follows from Remark~\ref{homotopy}, the second from adjunction, the third from~\eqref{eqn:injective}, and the last from adjunction again, along with the observation that $H^0(M) \otimes_{A^0} H^0(A) \cong H^0(M)$. Applying~\eqref{eqn:injective} once more proves~(1). For (2), we first observe that (1) implies that the functor $H^0 \co \cI \to \on{Inj}(H^0(A))$ is fully faithful. A bigraded analogue of \cite[Lemma 3.14]{minamoto} implies that every injective $H^0(A)$-module is of the form $\uHom_{A^0}(H^0(A), K)$ for some injective $A^0$-module $K$; applying \eqref{eqn:injective} therefore finishes the proof. 
\end{proof}

We now give a bigraded version of Minamoto's notion of minimal injective resolutions of dg-modules, which Minamoto calls ``minimal ifij resolutions" in \cite{minamoto}. Before stating it, we fix some notation: given a dg-$A$-module $M$, we write
\begin{equation}
\label{eqn:supinf}
\sup(M) = \sup\{i \text{ : } H^i(M) \ne 0\} \quad \text{and} \quad \inf(M) = \inf\{i \text{ : } H^i(M) \ne 0\}.
\end{equation}

 \begin{dfn}[\cite{minamoto} Definition 3.19] 
 \label{def:mininj}
 Let $A$ be as in Setup~\ref{setup}, and let $M$ be a dg-$A$-module such that $M$ is not exact and $H^i(M) = 0$ for $i \ll 0$. A \emph{minimal injective resolution of $M$} is a sequence of exact triangles
 $
\left[M_i \xra{f_i} I_i \xra{g_i} M_{i - 1} \to \right]
 $
in $\D(A)$ for $i \le 0$ such that $M_0 = M$, and for all $i \le 0$:
\begin{enumerate}
\item We have $I_i[\inf(I_i)] \in \cI$, and $\inf(I_i) = \inf(M_i)$.
\item The map $f_i \co H^{\inf(M_i)}(M_i) \to H^{\inf(M_i)}(I_i)$ is a graded $H^0(A)$-linear injective envelope. 
\end{enumerate}
 \end{dfn}

Since every graded module over an ordinary graded algebra admits an injective hull, it follows from \Cref{thm:minamoto} that every object $M$ as in Definition~\ref{def:mininj} admits a minimal injective resolution.

\begin{remark}
\label{rem:inf}
Let $A$ and $M$ be as in Definition~\ref{def:mininj}, and let $
\left[M_i \xra{f_i} I_i \xra{g_i} M_{i - 1} \to \right]
 $ be a minimal injective resolution of $M$. It is observed in \cite[Definition 3.19]{minamoto} that $\inf(I_i) \ge \inf(M)$ for all $i \le 0$. 
\end{remark}

\section{Balanced dualizing dg-modules}\label{section:3}

Yekutieli introduced in \cite{YEKUTIELI1_dualizing} the notion of a balanced dualizing complex over a noncommutative graded algebra. The main goal of this section is to introduce a generalization of balanced dualizing complexes in the setting of dg-algebras and to establish some cases in which they exist. We also introduce a notion of local cohomology for differential bigraded algebras; see also work of Shaul \cite{shaullocal, shaullocal2} for a similar dg-version of local cohomology. In addition, we prove a dg-version of the local duality theorem, which is a generalization of Yekutieli's noncommutative local duality theorem~\cite[Theorem 4.18]{YEKUTIELI1_dualizing}.

\subsection{Local cohomology of dg-modules}
\label{sec:lc}
Let $A$ be as in Setup~\ref{setup} and $\m$ its homogeneous maximal ideal. Given $M \in \D(A)$, we define $\RR\Gamma_\m(M) \ce \underset{d \to \infty}{\colim} \text{ }\RHom_A(A/A_{\ge d}, M)$, the \emph{derived $\m$-torsion module of $M$}. The $i^{\th}$ cohomology module $H^i_\m(M) \ce  \underset{d \to \infty}{\colim}\text{ } \uExt^i_A(A/A_{\ge d}, M)$ of $\RR\Gamma_\m(M)$ is called the $i^{\th}$ \emph{local cohomology of $M$} (with respect to $\m$). 

Let $\D^{\Tors}(A)$ be the full subcategory of $\D(A)$ given by objects $T$ such that every class $t \in H(T)$ satisfies $t \cdot H^0(A)_{\ge n} = 0$ for $n \gg 0$. We let $\D^{\tors}(A)$ be the full subcategory of  $\Db(A)$ given by objects that are also in $\D^{\Tors}(A)$. The functor $\RR\Gamma_\m$ clearly takes values in $\D^{\Tors}(A)$.

\begin{prop}
\label{prop:gen}
 The functor $\RR\Gamma_\m$ is the right adjoint of the inclusion $\D^{\Tors}(A) \into \D(A)$.
\end{prop}

\begin{proof}
Let $M \in \D(A)$. The surjections $A \to A / A_{\ge d}$ induce 
maps $\RHom_A(A / A_{\ge d}, M) \to M$ for all $d \ge 0$; passing to the colimit yields a natural morphism \begin{equation}
\label{eqn:naturalmap}
\RR\Gamma_\m(M) \to M.
\end{equation}
For any $T \in \D^{\Tors}(A)$, the morphism \eqref{eqn:naturalmap} induces a natural map
\begin{equation}
\label{eqn:adj}
\Hom_{\D(A)}(T, \RR\Gamma_\m(M)) \to \Hom_{\D(A)}(T, M).
\end{equation}
It suffices to show that ~\eqref{eqn:adj} is an isomorphism. Since $T$ is a filtered colimit of objects in $\D^{\tors}(A)$, we may assume $T \in \D^{\tors}(A)$. Moreover, since $\D^{\tors}(A)$ is generated by $\k(i)$ for all $i \in \Z$~\cite[Proposition 3.1]{brown2023orlovs}, we may further reduce to the case where $T=\k$. We now define an inverse 
$
\Hom_{\D(A)}(\kk, M) \to \Hom_{\D(A)}(\kk, \RR\Gamma_\m(M)) 
$
of the map~\eqref{eqn:adj}. Choose a $K$-injective resolution $I$ of $M$.  
Let $\phi \co \uHom_A(\k, I) \to \underset{d \to \infty}{\colim} \text{ }\uHom_A(A/A_{\ge d}, I)$ denote the canonical map. Since the image of $\phi$ is annihilated by $\m$, $\phi$ determines a chain map
$
\uHom_A(\k, I) \to \uHom_A(\k, \underset{d \to \infty}{\colim} \text{ }\uHom_A(A/A_{\ge d}, I)),
$
which yields a map
$
\Hom_{K(A)}(\k, I) \to \Hom_{K(A)}(\k, \underset{d \to \infty}{\colim} \text{ }\uHom_A(A/A_{\ge d}, I)),
$
and hence, via Remark~\ref{homotopy}, a map
\begin{equation}
\label{eqn:inverse}
\Hom_{\D(A)}(\k, M) \to \Hom_{\D(A)}(\k, \RR\Gamma_\m(M)).
\end{equation}
The map~\eqref{eqn:inverse} gives an inverse to  \eqref{eqn:adj} when $T = \k$.  Indeed, let $\alpha \in \Hom_{K(A)}(\k, \underset{d \to \infty}{\colim} \text{ }\uHom_A(A/A_{\ge d}, I))$. Since $\alpha(1)$ is annihilated by $\m$, we may represent it as an element $\{\alpha_d\}_{d \ge 1} \in \bigoplus_{d \ge 1} \uHom_A(A/A_{\ge d}, I)$ such that $\alpha_d = 0$ for $d \ne 1$. The map \eqref{eqn:adj} sends $\alpha$ to $\alpha_1$. On the other hand, the map \eqref{eqn:inverse} sends an element $\beta \in \Hom_{K(A)}(\k, I)$ to the map in $\Hom_{K(A)}(\k, \underset{d \to \infty}{\colim} \text{ }\uHom_A(A/A_{\ge d}, I))$ that sends $1$ to the element of $\underset{d \to \infty}{\colim} \text{ }\uHom_A(A/A_{\ge d}, I)$ represented by $\{\beta_d\}_{d \ge 1} \in \bigoplus_{d \ge 1} \uHom_A(A/A_{\ge d}, I)$, where $\beta_1 = \beta$ and $\beta_d = 0$ for~$d \ne 1$. One easily checks that these maps are inverses.
\end{proof}

\begin{remark}
\label{rem:counit}
It follows from the proof of Proposition~\ref{prop:gen} that the map \eqref{eqn:naturalmap} is the counit of the adjunction between the inclusion $\D^{\Tors}(A) \into \D(A)$ and $\dGamma_\m$. 
\end{remark}

If $M\in \D(A^{\op})$, denote its derived $\m$-torsion module and local cohomology by $\RR\Gamma_{\m^{\op}}(M)$ and $H^*_{\m^{\op}}(M)$.

\begin{remark}
\label{rem:bimodule}
Let $M$ be a dg-$A$-$A$-bimodule. Applying Proposition~\ref{prop:injbimodule}, we may choose a $K$-injective resolution of $M$ as an $A^{\op} \otimes_\k A$-module that is also $K$-injective as a dg-$A$-module and a dg-$A^{\op}$-module. For all $d \in \Z$, the objects $\RHom_A(A / A_{\ge d}, M)$ and $\RHom_{A^{\op}}(A^{\op} / A^{\op}_{\ge d}, M)$ are dg-$A$-$A$-bimodules; it follows that both $\RR\Gamma_{\m}(M)$ and $\RR\Gamma_{\m^{\op}}(M)$ are dg-$A$-$A$-bimodules as well. Moreover, the counit map \eqref{eqn:naturalmap} is a morphism of dg-$A$-$A$-bimodules. 
\end{remark}

The following result plays no role in what follows, but we include it as it may be of independent interest:

\begin{prop}
Let $A$ be as in Setup~\ref{setup}, $\m$ its homogeneous maximal ideal, and $\n$ the homogeneous maximal ideal of $A^0$. Given a dg-$A$-module $M$, there is an isomorphism $\dGamma_\m(M) \cong \dGamma_\n(M)$ in $\D(A^0)$. 
\end{prop}

\begin{proof}
The proof of \cite[Proposition 3.6]{brown2023orlovs} implies that the extension of scalars functor $F \co \D(A^0) \to \D(A)$ sends $\D^{\Tors}(A^0)$ to $\D^{\Tors}(A)$. We therefore have a commutative square
$$
\xymatrix{
\D^{\Tors}(A^0) \ar[r] \ar[d]^-{F} & \D(A^0) \ar[d]^-{F} \\
\D^{\Tors}(A) \ar[r] & \D(A),
}
$$
where the horizontal maps are inclusions. Of course, $F$ is the left adjoint of the restriction of scalars functor~$G$, and it follows from Proposition~\ref{prop:gen} that the horizontal maps are left adjoints of the derived $\n$-torsion and $\m$-torsion functors; passing to right adjoints, we thus have a commutative square
$$
\xymatrix{
\D^{\Tors}(A^0)  & \ar[l]_-{\RR\Gamma_\n} \D(A^0)  \\
\D^{\Tors}(A) \ar[u]_-{G}& \D(A) \ar[l]_-{\RR\Gamma_\m} \ar[u]_-{G}.
}
$$
In particular, both  $G \circ \dGamma_\m$ and $\dGamma_\n \circ G$ are right adjoints to the functor $\D^{\Tors}(A^0) \xra{F} \D^{\Tors}(A) \to \D(A)$. We therefore have an isomorphism $\RR\Gamma_\m(M) \xra{\cong} \RR\Gamma_\n(M)$ in $\D(A^0)$.
\end{proof}

\subsection{Local Duality}
We now prove a dg-version of local duality, which is a generalization of Yekutieli's noncommutative local duality theorem~\cite[Theorem 4.18]{YEKUTIELI1_dualizing}. We begin with a bigraded version of Frankild-Iyengar-J\o rgensen's notion of a dualizing dg-module. Before stating it, we recall that, given a dg-algebra $A$ and a dg-$A$-$A$-bimodule $R$, a dg-$A$-module $M$ is called \emph{$R$-reflexive} if the natural map
$$
M \to \RHom_{A^{\op}}(\RHom_A(M, R), R)
$$
is a quasi-isomorphism.
\begin{defn}[cf. \cite{FIJ} Section 1.8]
\label{def:dualizing}
Let $A$ be as in Setup~\ref{setup}. A dg-$A$-$A$-bimodule $R$ is called \emph{dualizing} if, for any $M \in \Db(A^{\op})$ and $N \in \Db(A)$, the following conditions are satisfied:
\begin{enumerate}

\item We have $\RHom_{A^{\op}}(M, R) \in \Db(A)$, and $\RHom_A(N, R) \in \Db(A^{\op})$.
\item The dg-$A$-modules $N$ and $N \otimes_A R$ and dg-$A^{\op}$-modules  $M$ and $M \otimes_{A^{\op}} R$ are $R$-reflexive.

\end{enumerate}
\end{defn}

\begin{remark}
The definition of a dualizing dg-module in \cite[Section 1.8]{FIJ} requires the existence of a biprojective and biinjective resolution of $R$: see \cite[Section 1.2]{FIJ} for the definitions of these objects. But such resolutions always exist in our setting: the existence of biprojective resolutions is standard since our ground ring $\k$ is a field (see e.g. \cite[Proposition 1.3]{FIJ}), and the existence of biinjective resolutions is the content of Proposition~\ref{homotopy_flat}(2). Additionally, the requirement in the definition of a dualizing dg-module in \cite[Section 1.8]{FIJ} that $A$ is $R$-reflexive as both a dg-$A$-module and a dg-$A^{\op}$-module is superfluous in our setting: it follows from (2) in Definition~\ref{def:dualizing}.
\end{remark}

\begin{example}
\label{rem:dualizing}
If $A$ is Gorenstein (Definition~\ref{def:gorenstein}), then $A(i)[j]$ is dualizing for all $i, j \in \Z$.  
\end{example}

Let $A$ be as in Setup~\ref{setup}. If $A$ admits a dualizing dg-module $R$ (Definition~\ref{def:dualizing}), then we have the following duality functors, which are inverse equivalences:
\begin{align*}
D & \c \Db(A) \to \Db(A^{\op})^{\op}, \quad D(M) = \RHom_A(M, R);  \\
D^{\op} & \c \Db(A^{\op})^{\op} \to \Db(A), \quad D^{\op}(M) =\RHom_{A^{\op}}(M, R).
\end{align*}

\begin{thm}[Local Duality, cf. \cite{YEKUTIELI1_dualizing} Theorem 4.18]\label{thm: local_duality}
    Let $A$ be as in Setup~\ref{setup}, and assume $A$ admits a dualizing dg-module $R$. Let $M\in \Db(A)$. There is a canonical isomorphism in $\D(A)$
    \[
    \dGamma_{\m}(M)\cong \RHom_{A^{\op}}(D(M),\dGamma_{\m}(R)).
    \]
    
\end{thm}

See also \cite[Theorem 7.26]{Shaul_injective} for a version of local duality involving differential $\Z$-graded algebras (as opposed to  differential bigraded algebras, which we consider here).

\begin{proof}

    Fix $n \ge 1$. By Remark~\ref{rem:bimodule}, the canonical map $\RR\Gamma_\m(R) \to R$ given by~\eqref{eqn:naturalmap} is a morphism of $A$-$A$-bimodules, and it induces a map
    $\a_n \co \RR\uHom_A(A/A_{\ge n}, \dGamma_{\m}(R)) \to D(A/A_{\ge n})$
    of $A$-$A$-bimodules, which is a quasi-isomorphism due to the adjunction in Proposition~\ref{prop:gen}. We therefore have a morphism $D(A/A_{\ge n}) \to \dGamma_{\m}(R)$ in $\D(A^{\op} \otimes_\k A)$ given by the diagram
\begin{equation}
\label{eq:nmap}
D(A/A_{\ge n}) \xleftarrow[\simeq]{\a_n} \RR\uHom_A(A/A_{\ge n}, \dGamma_{\m}(R)) \to  \dGamma_{\m}(R),
\end{equation}
where the second map is induced by the surjection $A \onto A/A_{\ge n}$. Let $f_n$ denote the composition
$$
\RHom_A(A/A_{\ge n},M) \xra[\simeq]{\gamma_n} \RHom_{A^{\op}}(D(M),D(A/A_{\ge n})) \xra{\beta_n} \RHom_{A^{\op}}(D(M), \dGamma_{\m}(R)),
$$
where $\gamma_n$ is the $A$-linear quasi-isomorphism induced by the equivalence $D$, and $\beta_n$ is induced by \eqref{eq:nmap}. 
Since filtered colimits are exact in $\on{Mod}(A)$, the colimit of the maps $\gamma_n$ is an isomorphism in $\D(A)$. The colimit of the maps $\beta_n$ is also  an isomorphism in $\D(A)$, since $D(M) \in \Db(A^{\op})$. We conclude that the colimit of the maps $f_n$ gives the desired isomorphism in $\D(A)$. 
\end{proof}

\subsection{Balanced dualizing dg-modules}
\label{sec:dualizing}

The following definition is a direct analogue of Yekutieli's notion of a balanced dualizing complex:

\begin{defn}[cf. \cite{YEKUTIELI1_dualizing} Definition 4.1]\label{def:balanced_dualizing} Let $A$ be a dg-algebra and $R$ a dg-$A$-$A$-bimodule. We say $R$ is a \emph{balanced dualizing dg-module} if it is a dualizing dg-module in the sense of Definition~\ref{def:dualizing}, and there are isomorphisms $\RR\Gamma_{\m}(R) \cong \uHom_{\k}(A, \k) \cong \RR\Gamma_{\m^{\op}}(R)$ in $\D(A^{\op} \otimes_\k A)$.
\end{defn}

We record the following consequence of local duality (Theorem~\ref{thm: local_duality}):
\begin{cor}
\label{rem:LD}
If $A$ is as in Setup~\ref{setup}, and $M \in \Db(A)$, then there is an isomorphism $\dGamma_{\m}(M)\cong \RHom_{A^{\op}}(D(M),\uHom_{\k}(A, \k))$ in $\D(A)$.
\end{cor}

When $A$ is as in Setup~\ref{setup}, and $A = A^0$, our definition agrees with Yekutieli's definition of a balanced dualizing complex~\cite[Definition 4.1]{YEKUTIELI1_dualizing}. In this case, Yekutieli proves that balanced dualizing complexes exist for AS-regular algebras~\cite[Corollary 4.14]{YEKUTIELI1_dualizing} (in fact, the proof of this statement only requires that the algebra is Gorenstein, in the sense of Definition~\ref{def:gorenstein}), finitely generated $\k$-algebras that are finitely generated over their centers~\cite[Corollary 5.6]{{YEKUTIELI1_dualizing}}, and skew homogeneous coordinate rings~\cite[Theorem 7.3]{YEKUTIELI1_dualizing}. We now verify the existence of a balanced dualizing dg-module in some examples.

\begin{prop}\label{lem:fd_balanced}
    Let $A$ be as in Setup~\ref{setup}, and assume $\dim_\k H(A) < \infty$. The dg-$A$-$A$-bimodule $\uHom_\k(A,\k)$ is a balanced dualizing dg-module for $A$.
\end{prop}

\begin{proof}
Set $R \ce \uHom_\k(A,\k)$, and let $M\in \Db(A^{\op})$ and $N\in \Db(A)$. It follows by adjunction that $\RHom_{A^{\op}}(M,R)$ has finitely generated cohomology over $H^0(A)$, and so it is an object in $\Db(A)$; similarly, $\RHom_A(N,R) \in \Db(A^{\op})$. The natural map $N\to \RHom_{A^{\op}}(\RHom_A(N,R),R)$ in $\D(A)$ is the composition of the isomorphism $N \xra{\cong} \RHom_\k(\RHom_\k(N,\k),\k)$ and the adjunction isomorphism given by $\RHom_\k(\RHom_\k(N,\k),\k)\xra{\cong} \RHom_{A^{\op}}(\RHom_A(N,R),R)$; thus $N$ is $R$-reflexive. One similarly verifies that $M$ is $R$-reflexive. This proof of reflexivity does not require $N$ to be in $\Db(A)$, but only that $\dim_\k H^j(N)_i < \infty$ for all $i, j$ (and similarly for $M$). Thus, $N \otimes_A R$ and $M \otimes_{A^{\op}} R$ are also $R$-reflexive, and so $\uHom_\k(A,\k)$ is a dualizing dg-module for $A$. Finally, the counit maps $\RR\Gamma_{\m}(R) \to R$ and $\RR\Gamma_{\m^{\op}}(R) \to R$ are both morphisms of dg-$A$-$A$-bimodules, and they are isomorphisms in $\D(A^{\op} \otimes_\k A)$ since $\dim_\k H(R) < \infty$. 
\end{proof}

\begin{ex}
\label{ex:jin}
    If $A$ is a (possibly noncommutative) trivial extension dg-algebra, as defined in \cite[Section 6]{JIN2020}, then $A$ satisfies the conditions in Setup~\ref{setup}, and $\dim_\kk H(A) < \infty$. Proposition~\ref{lem:fd_balanced} therefore implies that $\uHom_\k(A, \k)$ is a balanced dualizing dg-module for $A$. 
\end{ex}

\begin{remark}
\label{rem:fd}
Suppose $A$ is as in Proposition~\ref{lem:fd_balanced}, and assume $A$ is also Gorenstein. Choose $a, n \in \Z$ such that there is an isomorphism $\RHom_A(\k, A) \cong \k(a)[-n]$ in $\D(A)$. We have an isomorphism of dg-$A$-$A$-bimodules:
$$
\uHom_\k(A, \k) \cong \uHom_\k(A, \RHom_A(\k, A))(-a)[n] \cong A(-a)[n]. 
$$
\end{remark}

\begin{thm}[cf. \cite{YEKUTIELI1_dualizing} Proposition 4.4]
\label{prop:balanced}
Suppose $A$ is graded commutative and Gorenstein (see Definition~\ref{def:gorenstein}). Choose $a, n \in \Z$ such that $\RHom_A(\k, A) \cong \k(a)[-n]$ in $\D(A)$. The dg-$A$-$A$-bimodule $R \ce A(-a)[n]$ is a balanced dualizing dg-module.
\end{thm}

\begin{proof}
As observed in \Cref{rem:dualizing}, $R$ is a dualizing dg-module for $A$. We refer the reader to \cite[Definition 3.1]{minamoto} for the notion of injective dimension of a dg-module $N$ (which is due to Yekutieli~\cite{yekdual}). Given a dg-$A$-module $N$, we denote its injective dimension by $\id_A(N)$. By \cite[Theorem 3.9]{Minamoto2}\footnote{Throughout this proof, we reference results in \cite{Minamoto2} and \cite{minamoto}, although Minamoto does not work with dg-algebras with an internal grading. However, the bigraded versions of the referenced results hold verbatim, by the same proofs.} , the injective dimension of $A$ over itself is the Krull dimension of $H^0(A)$. We thus have $\id_A(R) = \id_A(A) = \dim H^0(A) <\infty$. Given a homogeneous prime ideal $\mathfrak{p}$ in $H^0(A)$ and a dg-$A$-module $N$, there is a notion of Bass numbers $\mu^n(\mathfrak{p}, N)$ due to Minamoto; we refer to \cite[Section 2.3.2]{Minamoto2} for the definition. Combining \cite[Lemma 2.26]{Minamoto2} and \cite[Proposition 2.29]{Minamoto2}, we have that $\mu^n(\mathfrak{p}, R) = 0$ for $n \gg 0$. It therefore follows from \cite[Theorem 1.1]{Minamoto2} and Remark~\ref{rem:inf} that $M$ admits a minimal injective resolution $\left\{
\left[M_i \xra{f_i} I_i \xra{g_i} M_{i - 1} \to \right]
 \right\}_{i \le 0}$, as in Definition~\ref{def:mininj}, such that $I_{-i} = 0$ for $i \gg 0$. Let $e \ce \max\{i \text{ : } I_{-i} \ne 0\}$.

Our next goal is---roughly speaking---to show that $R$ admits a $K$-injective resolution $J$ built out of the above minimal injective resolution of $R$.
Since the dg-module $R$ satisfies the conditions in \cite[Definition 3.1]{Minamoto2}, it follows from \cite[Theorem 3.5]{Minamoto2} that $e = \id_A(R)=\dim H^0(A)$. We observe that, by \cite[Theorem 3.5]{Minamoto2} and \cite[Definition 3.1]{minamoto}, we have $\inf(M_{-i})=\inf(I_{-i})=\inf(R)$ for all $i=0,\dots,e$. If $\id_A(R)=0$, then $\dim H^0(A)=0$, and hence $\dim_{\kk}H(A)<\infty$; the statement holds in this case by Proposition~\ref{lem:fd_balanced} and \Cref{rem:fd}. Assume $\id_A(R)\geq 1$. Applying \cite[Lemma 3.22]{minamoto} and induction, we have that $\id_A(M_{-e})=0$. By \cite[Lemma 3.17]{minamoto}, we conclude that there is an isomorphism $M_{-e} \cong I[l]$ in $\D(A)$ for some $I\in \mathcal{I}$ and $\ell \in \Z$. By Remark~\ref{rem:semiinj}, $I[l]$ is $K$-injective. Moreover, it follows from \cite[Theorem 3.10(4)]{minamoto} that $M_{-e}\to I_{-e}$ is an isomorphism in $\D(A)$, and hence we may assume that $I[l]=I_e$. Using the triangle $M_{-e+1}\to I_{-e+1}\to M_{-e} \to $, one may construct a $K$-injective resolution of $M_{-e+1}$ from $I_{-e+1}$ and $I[l]$. Continuing inductively gives a $K$-injective resolution $R \xra{\simeq} J$ that is built from the data of our minimal injective resolution of $R$. 

We have isomorphisms $\kk \cong \RHom_A(\kk,R) \cong \lhom_A(\kk,J)$ in $\D(A)$. For all $i\leq 0$, we have isomorphisms $I_i \cong \lhom_{A^0}(A,K_i)[-\inf(R)]$ in $\D(A)$ for some $K_i\in \Inj(A^0)$. It therefore follows from adjunction that, for all $i\leq 0$, we have $\lhom_A(\kk,I_i)\cong \lhom_{A^0}(\kk,K_i)[-\inf(R)]$ as bigraded complexes of vector spaces. In particular, $\lhom_A(\kk,I_i)$ lives in the single cohomological degree $j \ce \inf(R)$. For all $i, q \in \mathbb{Z}$, we thus have $\lhom_A(\kk,I_i)_q \cong \Hom_{\D(A)}(\kk,I_i(q)[j])$. \Cref{thm:minamoto}(1) and \cite[Corollary 3.29]{minamoto} imply that $\Hom_{\D(A)}(\kk,I_i(q)[j]) \cong \Hom_{H^0(A)}(\kk, H^0\lhom_{A^0}(A,K_i(q))) 
                                    \cong \Hom_{\D(A)}(\kk, R(q)[j-i])$.
This last object is zero unless $q=0$ and $i=j$. If $\lhom_A(\kk,I_i)=0$ for all $i$, it follows that $\lhom_A(\kk,J)=0$, a contradiction. Thus, we have $-e \le j \le 0$, and  $\Hom_A(\kk, J) \cong \lhom_A(\kk,I_j)\cong \kk$ in $\Mod(A)$. 

It suffices to exhibit an isomorphism $\Gamma_{\m}(J)\ce \underset{d \to \infty}{\colim}\lhom_A(A/A_{\geq d},J)\cong \lhom_\k(A,\k)$ as dg-$A$-modules. We note that $\lhom_{\kk}(A,\kk)$ is an essential extension of $\kk$ (see \cite[Section 3.10]{popescu} for the definition) in $\Mod(A^{\nat})$ and hence also in $\Mod(A)$. By Remark~\ref{rem:semiinj}, $\lhom_{\kk}(A,\kk)\cong \lhom_{A^0}(A, \lhom_{\kk}(A^0,\kk))\in \mathcal{I}$ is an injective object in $\Mod(A)$. By \cite[Lemmas 10.5 and 10.6]{popescu}, $\lhom_{\kk}(A,\kk)$ has no proper essential extensions in $\Mod(A)$ and is hence an injective envelope for $\kk$ in $\Mod(A)$. The underlying bigraded module of $\Gamma_{\m}(J)$ coincides with the $\m$-torsion submodule of $J$ over $A^{\nat}$; thus, $\Gamma_{\m}(J)$ is essential over $\lhom_A(\kk,J)\cong \kk$ in $\Mod(A)$. By \cite[Proposition 10.7 and Lemma 10.8]{popescu}, there is an injective morphism $f \co \Gamma_{\m}(J)\rightarrow \lhom_\k(A,\k)$ in $\Mod(A)$. Since each $I_i$ is a bigraded injective $A^{\nat}$-module, so is $\Gamma_{\m}(I_j)$. For each $i\neq j$, it follows that $\Gamma_{\m}(I_i)=0$, since it is essential over $\lhom_A(\kk,I_i)=0$. We therefore have isomorphisms $\Gamma_{\m}(J)\cong \Gamma_{\m}(I_j)\cong E_{A^{\nat}}(\kk)\cong \lhom_{\kk}(A,\kk)$ in $\Mod(A^{\nat})$. Since each bigraded component of $\lhom_{\kk}(A,\kk)$, and hence $\Gamma_{\m}(J)$, is a finite dimensional $\k$-vector space, it follows that $f \co \Gamma_{\m}(J)\to \lhom_{\kk}(A,\kk)$ is an isomorphism in $\Mod(A)$.
\end{proof}

\begin{remark}
Let $A$ and $n$ be as in Theorem~\ref{prop:balanced}. The inequality $j \ce \inf(R) \le 0$ in the proof of Theorem~\ref{prop:balanced} implies that $n \ge \inf(A)$. 
\end{remark}

\begin{ex}
\label{ex:bal1}
 Theorem~\ref{prop:balanced} applies to each of the families of dg-algebras in \cite[Examples 2.23 - 2.26]{brown2023orlovs}. For instance, Theorem~\ref{prop:balanced} applies to the generalized Koszul complexes described in \cite[Example 2.23]{brown2023orlovs} and originally studied by Shaul in \cite{Shaul_Koszul}.
\end{ex}

\begin{remark}
\label{rem:comm}
    Our proof of \Cref{prop:balanced} uses in an essential way Minamoto's structural results for minimal injective resolutions developed in~\cite{Minamoto2}. Since those results require graded commutativity, our argument does as well. A suitable theory of minimal injective resolutions for noncommutative dg-algebras 
    therefore seems to be needed to prove Theorem~\ref{prop:balanced} without the commutativity assumption. 
\end{remark}

We conjecture that the graded commutativity hypothesis can be removed from Theorem~\ref{prop:balanced}. 

\begin{conj}
\label{conj:balanced}
Let $A$ be a Gorenstein dg-algebra (see Definition~\ref{def:gorenstein}). Choose $a, n \in \Z$ such that there is an isomorphism $\RHom_A(\k, A) \cong \k(a)[-n]$ in $\D(A)$. The dg-algebra $A$ admits a balanced dualizing dg-module~$R$ that is isomorphic in $\D(A)$ to $A(-a)[n]$.
\end{conj}

A positive answer to Conjecture~\ref{conj:balanced} would give a generalization of Yekutieli's result in \cite{YEKUTIELI1_dualizing} that noncommutative Gorenstein $\k$-algebras admit balanced dualizing complexes. As evidence for Conjecture~\ref{conj:balanced}, apart from  Proposition~\ref{lem:fd_balanced} and Theorem~\ref{prop:balanced} (along with Remark~\ref{rem:fd}), we observe that, if $A$ and $R$ are as in Conjecture~\ref{conj:balanced}, then $\RR\Gamma_\m(R)$ and $\uHom_\k(A, \k)$ are quasi-isomorphic complexes of $\k$-vector spaces. Indeed, recall that there is an isomorphism $\RHom_A(\k, R) \cong \kk$ in $\D(A)$. Applying $\RHom_A( -, A)$ to the short exact sequence
$
0 \to A_{\ge 1} / A_{\ge 2} \to A / A_{\ge 2} \to \kk \to 0
$
yields that $\uExt^*_A(A / A_{\ge 2}, A)$ is $\kk$-linearly isomorphic to $H(\uHom_\kk(A, \k))_0 \oplus H(\uHom_\kk(A, \k))_{-1}$. Continuing in this way, we conclude by induction that $\uExt^*_A(A / A_{\ge d}, A)$ is isomorphic as a bigraded $\k$-vector space to $\bigoplus_{i = 0}^{d - 1} H(\uHom_\kk(A, \k))_{-i}$, and passing to the colimit gives the desired $\k$-linear quasi-isomorphism. The challenge is exhibiting a quasi-isomorphism between $\RR\Gamma_\m(R)$ and $\uHom_\k(A, \k)$ as dg-$A$-modules.

\section{Serre duality}
\label{sec:serreduality}

We now prove our main result, Theorem~\ref{thm: Serre_duality}, which is a generalization to dg-algebras of Yekutieli-Zhang's noncommutative Serre Duality Theorem (Theorem~\ref{thm:intro1}) 

\subsection{Noncommutative geometry over a dg-algebra}
\label{sec:qgr}

Let $A$ be as in Setup~\ref{setup}. As in the introduction (see also \cite[Section 3]{brown2023orlovs}), we let $\D_{\mathrm{Qgr}}(A)$ denote the quotient $\D(A)/\D^{\Tors}(A)$. We also let $\Dqgr(A)$ denote the essential image of the fully faithful embedding $\Db(A)/\D^{\tors}(A)\hookrightarrow \D_{\mathrm{Qgr}}(A)$. We recall from the introduction that, given $M \in \D(A)$, we let $\widetilde{M}$ denote the corresponding object in $\D_{\mathrm{Qgr}}(A)$. We have canonical triangulated functors
$$
\Pi : \D(A) \to \D_{\mathrm{Qgr}}(A) \quad \text{and}  \quad \pi : \Db(A) \to \Dqgr(A)
$$
given by $M \mapsto \widetilde{M}$. If $A = A^0$, and $A$ is commutative and standard graded, then $\D_{\mathrm{Qgr}}(A)$ (resp. $\Dqgr(A)$) is equivalent to the derived category of quasi-coherent (resp. coherent) sheaves on $\Proj(A)$. 

\begin{remark}
Suppose $A$ is as in Setup~\ref{setup}; in addition, assume $A^0$ is commutative and generated in degree~1, and also that $A$ is finitely generated as an $A^0$-module. In this case, $A$ determines a sheaf of dg-algebras $\mathcal{A}$ on $\Proj(A^0)$, and $\D(A)$ (resp. $\Db(A)$) may be interpreted as the derived category of quasi-coherent (resp. coherent) dg-$\A$-modules: see \cite[Section 3.2]{brown2023orlovs} for details. 
\end{remark}

The functor $\Pi : \D(A) \to \D_{\mathrm{Qgr}}(A)$ admits a fully faithful right adjoint~\cite[Proposition 3.3]{brown2023orlovs}, the \emph{derived global sections functor}, which we denote by $\RR\Gamma_*$.

\begin{prop}\label{prop:derived_glsections_description}
    If $A$ is as in Setup~\ref{setup}, and $M \in \D(A)$, then there is an isomorphism $\RR\Gamma_*(\widetilde{M})\cong \underset{d \to \infty}{\colim} \text{ }\RHom_A(A_{\ge d}, M )$ in $\D(A)$.
\end{prop}

\begin{proof}
Let $M \in \D(A)$. It follows from (the proof of) \cite[Proposition 4.9.1]{krause} that there is an exact triangle
\begin{equation}
\label{eqn:adjtri}
\dGamma_{\m}(M) \xrightarrow{} M \xrightarrow{} \RR\Gamma_*(\w{M}) \xrightarrow{} ,
\end{equation}
where the first morphism is the counit of the adjunction between the inclusion $\D^{\Tors}(A) \into \D(A)$ and the functor $\dGamma_{\m}$. 
On the other hand, for all $d \ge 0$, the exact triangle
$
A_{\ge d} \to A \to A / A_{\ge d} \to 
$
induces an exact triangle
$
\RHom_A(A/A_{\geq d},M) \xrightarrow{} M \xrightarrow{} \RHom_A(A_{\geq d},M) \xrightarrow{} ;
$
passing to colimits, we obtain the exact triangle
\begin{equation}
\label{eqn:adjtri2}
\dGamma_{\m}(M) \xrightarrow{} M \xrightarrow{} \underset{d \to \infty}{\colim}\text{ }\RHom_A(A_{\geq d},M) \xrightarrow{} .
\end{equation}
Finally, by Remark~\ref{rem:counit}, the first morphisms in the triangles \eqref{eqn:adjtri} and \eqref{eqn:adjtri2} coincide.
\end{proof}
Recall from the introduction that, given $\widetilde{M}, \widetilde{N} \in \D_{\Qgr}(A)$, we write $\Ext^j(\widetilde{M}, \widetilde{N}) \ce \Hom_{\D_{\Qgr}(A)}(\widetilde{M}, \widetilde{N}[j])$. Given $A$ as in Setup~\ref{setup} and $M \in \Mod(A)$, let $\RR^j\Gamma(\tM) \ce H^j\dGamma_*(\tM)_0$.

\begin{prop}
Let $A$ be as in Setup~\ref{setup}, $M \in \Db(A)$, and $N \in \D(A)$. For all $j \in \Z$, we have an isomorphism
$
\Ext^j(\widetilde{M}, \widetilde{N}) \cong \underset{d \to \infty}{\colim}\Ext^{j}_{A}(M_{\geq d},N).
$
In particular, we have $\RR^j\Gamma(\widetilde{N}) \cong \Ext^j(\widetilde{A}, \widetilde{N})$. 
\end{prop}

\begin{proof}
The triangle \eqref{eqn:adjtri} and the exactness of filtered colimits gives a long exact sequence
\begin{align*}
\cdots \to \underset{d \to \infty}{\colim} \Hom_{\D(A)}(M_{\ge d}, \RR\Gamma_\m(N)[j]) & \to  \underset{d \to \infty}{\colim} \Hom_{\D(A)}(M_{\ge d}, N[j])  \\ & \to \underset{d \to \infty}{\colim} \Hom_{\D(A)}(M_{\ge d}, \RR\Gamma_*(\widetilde{N})[j]) \\
& \to \underset{d \to \infty}{\colim} \Hom_{\D(A)}(M_{\ge d}, \RR\Gamma_\m(N)[j+1]) \to \cdots.
\end{align*}
Since $\widetilde{M_{\ge m}} \cong \widetilde{M}$ for all $m \in \Z$, adjunction implies that
$
\underset{d \to \infty}{\colim} \Hom_{D(A)}(M_{\ge d}, \RR\Gamma_*(\widetilde{N})[j])  \cong \Ext^j_A(\widetilde{M}, \widetilde{N}).
$
It therefore suffices to show that $\underset{d \to \infty}{\colim} \Hom_{\D(A)}(M_{\ge d}, \RR\Gamma_\m(N)[j]) = 0$ for all $j$. Replacing $N$ with an appropriate shift, we may assume without loss of generality that $j = 0$. Since $\RR\Gamma_\m(N)$ is torsion, it is a filtered colimit of objects in $\D^{\tors}(A)$; thus, since $M \in \Db(A)$, we need only show
$\underset{d \to \infty}{\colim} \Hom_{\D(A)}(M_{\ge d}, T) = 0$ for $T \in \D^{\tors}(A)$. There is a quasi-isomorphism $T \onto T / T_{\ge m}$ for $m \gg 0$, and so we may assume $T_m = 0$ for $m \gg 0$. Let $F^d$ be a semifree resolution of $M_{\ge d}$ as in \cite[Proposition 2.16]{brown2023orlovs} for all $d$; the construction of $F^d$ shows that $F^d_i = 0$ for $i < d$. We therefore have $\Hom_{\D(A)}(M_{\ge d}, T) \cong \Hom_{K(A)}(F^d, T) = 0$ for $d \gg 0$, where the isomorphism follows from Remark~\ref{homotopy}. This proves the first statement, and the second is immediate from the first, along with Proposition~\ref{prop:derived_glsections_description}. 
\end{proof}

\subsection{Serre duality}
We now prove Theorem~\ref{thm: Serre_duality}. We begin with the following technical result:

\begin{prop}
\label{prop:ext}
Let $A$ be a dg-algebra as in Setup~\ref{setup} with a balanced dualizing dg-module $R$, and let $M\in \Db(A)$. For $j\in \mathbb{Z}$ and $d \ge 1$, we have a canonical isomorphism of $\k$-vector spaces
$$
\dGamma^j(\tM)\cong \Hom_\k(\Ext_A^{-j-1}(M_{\ge d},R), \k).
$$

\end{prop}

\begin{proof}
Since $\widetilde{M_{\ge d}} = \tM$, the triangle \eqref{eqn:adjtri} yields a long exact sequence
$$
\cdots \to H^j(M_{\ge d})_0 \to \RR^j\Gamma(\widetilde{M}) \to H^{j+1}_\m(M_{\ge d})_0 \to H^{j+1}(M_{\ge d})_0 \to \cdots.
$$
Since $d \ge 1$, the outer terms vanish, yielding the isomorphism $\RR^j\Gamma(\widetilde{M}) \xra{\cong} H^{j+1}_\m(M_{\ge d})_0$. We therefore have isomorphisms
\begin{align*}
       \RR^j\Gamma(\widetilde{M})  \cong   H^{j+1}_\m(M_{\ge d})_0 
                   \cong   \Ext^{j+1}_{A^{\op}}(D(M_{\geq d}),\lhom_\k(A,\k)) 
                   & \cong   \Ext^{j+1}_{\k}(D(M_{\geq d}),\k) \\
                  & \cong   \Hom_\k(\Ext^{-j-1}_{A}(M_{\geq d},R),\k),
    \end{align*}
where the second isomorphism follows from Corollary~\ref{rem:LD}, the third from adjunction, and the fourth from the exactness of taking $\k$-duals. 
\end{proof}

\begin{proof}[Proof of Theorem~\ref{thm: Serre_duality}]
The canonical isomorphisms $\RR^j\Gamma(\tM)\cong \Hom_\k(\Ext_A^{-j-1}(M_{\ge d},R), \k)$ from Proposition~\ref{prop:ext} for $d \ge 1$ are compatible with the maps
$$
\Hom_\k(\Ext_A^{-j-1}(M_{\ge d+1},R), \k) \to \Hom_\k(\Ext_A^{-j-1}(M_{\ge d},R), \k)
$$
induced by the inclusions $M_{\ge d+1} \into M_{\ge d}$. Passing to the colimit as $d$ tends to infinity yields the desired isomorphism. 
\end{proof}

\begin{example}
\label{ex:gorserre}
If $A$ is a graded commutative and Gorenstein dg-algebra, then Theorem~\ref{prop:balanced} implies that $A(-a)[n]$ is a balanced dualizing dg-module for $A$. It therefore follows from Theorem~\ref{thm: Serre_duality} that, for all $M \in \Db(A)$ and $j \in \Z$, we have a $\k$-linear isomorphism
$$
\RR^j\Gamma(\widetilde{M}) \cong \Hom_\k(\Ext^{n-j-1}(\widetilde{M}, \widetilde{A(-a)})).
$$
For instance, this isomorphism holds for each of the families of dg-algebras in \cite[Examples 2.23---2.26]{brown2023orlovs}. A positive answer to Conjecture~\ref{conj:balanced} would imply that the graded commutative assumption in this example may be removed.
\end{example}

\section{Application to finiteness properties of derived global sections}
\label{sec:app}
The goal of this section is to apply our Serre Duality Theorem (Theorem~\ref{thm: Serre_duality}) to establish some finiteness properties for derived global sections of objects in $\D_{\qgr}(A)$. Let us begin by fixing some notation. Given $i \in \Z$, let $\D(A)_{\ge i}$ denote the subcategory of $\D(A)$ given by objects $M$ such that $M_j = 0$ for $j < i$, and set $\Db(A)_{\ge i} \ce \Db(A) \cap \D(A)_{\ge i}$. The truncation functor $\tau_{\ge i} \co \D(A) \to \D(A)_{\ge i}$ is the right adjoint to the inclusion map $\D(A)_{\ge i} \into \D(A)$. We conclude that $\RR\Gamma_{\ge i} \ce \tau_{\ge i} \circ \RR\Gamma_*$ is the right adjoint of the  composition $\D(A)_{\ge i} \into \D(A) \xra{\Pi} \D_{\Qgr}(A)$.

\begin{thm}
\label{thm:cd}
Let $A$ be as in Setup~\ref{setup}. Assume $A$ is Gorenstein and that $A$ admits a balanced dualizing dg-module $R$. Let $M \in \Db(A)$, and fix $i \in \Z$. 
\begin{enumerate}
\item $H^j(\RR\Gamma_{\ge i}(\widetilde{M}))$ is a finitely generated $H^0(A)$-module for all $j \in \Z$.
\item 
$\RR^j\Gamma(\widetilde{M}) \ne 0$ only if $\inf(M)  \le j \le \sup(M) - \inf(R) - 1$ (see~\eqref{eqn:supinf} for the definitions of the $\sup$ and $\inf$ of a dg-module). 
\item $\RR\Gamma_{\ge i}$ maps objects in $\D_{\qgr}(A)$ to $\Db(A)_{\ge i}$. Thus, $\RR\Gamma_{\ge i}$ gives a fully faithful right adjoint of the composition $\Db(A)_{\ge i} \into \Db(A) \xra{\pi} \D_{\qgr}(A)$. 
\end{enumerate}
\end{thm}

\begin{example}
Suppose $A$ and $M$ are as in Theorem~\ref{thm:cd}. Choose $a, n \in \Z$ such that there is an isomorphism $\RHom_A(\kk, A) \cong \k(a)[-n]$ in $\D(A)$. If $A$ is graded commutative, then $A(-a)[n]$ is a balanced dualizing dg-module for $A$ by Theorem~\ref{prop:balanced}. Theorem~\ref{thm:cd}(2) therefore implies:
$$
\RR^j\Gamma(\widetilde{M}) \ne 0 \quad \text{only if} \quad \inf(M)  \le j \le \sup(M) + n - \inf(A) - 1.
$$
Assume, in addition, that $A = A^0$, $A$ is generated in internal degree 1, and $M$ is concentrated in degree 0. Let $X = \Proj(A)$. In this case,  $n = \dim(A)$, and so Theorem~\ref{thm:cd} recovers the classical statement that the $A$-module $\bigoplus_{\ell \ge i} H^j(X, \widetilde{M}(\ell))$ is finitely generated for all $i$, and $H^j(X, \widetilde{M}) \ne 0$ only if $0 \le j \le n - 1 = \dim(X)$. 
\end{example}

It was previously known that the functor $\RR\Gamma_{\ge i}$ restricts to a functor $\D_{\qgr}(A) \to \Db(A)_{\ge i}$ for dg-algebras $A$ as in Setup~\ref{setup} in the following cases:
\begin{enumerate}
\item A result of Artin-Zhang implies this when $A = A^0$, and $A$ satisfies Artin-Zhang's \emph{condition~$\chi$} \cite[Theorem 3.8(3)]{AZ}.
\item By \cite[Proposition 3.6]{brown2023orlovs}, this holds when $A^0$ is either Gorenstein or commutative.
\end{enumerate}

As a first step toward proving Theorem~\ref{thm:cd}, we formulate a natural extension of Artin-Zhang's condition~$\chi$ to dg-algebras:

\begin{dfn}[cf. \cite{AZ} Definition 3.7]\label{def:chi}
Let $A$ be a dg-algebra as in Setup~\ref{setup}. We say $A$ satisfies \emph{condition $\chi$} if, for all $M\in \Db(A)$ and $i \in \Z$, we have $\uExt^i_A(\k, M)_j = 0$ for $j \gg 0$. 
\end{dfn}

Suppose $A$ is as in Setup~\ref{setup}. If $A = A^0$, then Definition~\ref{def:chi} is equivalent to Artin-Zhang's condition~$\chi$ by \cite[Proposition 3.8(1)]{AZ}. When $A$ is graded commutative, condition $\chi$ always holds: indeed, choosing a semifree resolution $F$ of $\k$ as in \cite[Proposition 2.16]{brown2023orlovs}, it is evident that $\Ext^j_A(\k, M) = H^j\uHom_A(F, M)$ is finitely generated over $A^0$; since it is also a $\k$-module, condition~$\chi$ follows immediately. This argument fails in the noncommutative case, since one cannot always equip the semifree resolution $F$ of $\k$ with a dg-$A$-$A$-bimodule structure. Indeed, there are algebras $A$ as in Setup~\ref{setup} with $A = A^0$ such that condition $\chi$ does not hold (see \cite[Example 0.1]{SZ}). The following generalization of a result of Yekutieli-Zhang \cite[Corollary 4.3]{Yekutieli-Zhang} shows that condition~$\chi$ holds for Gorenstein dg-algebras:

\begin{prop}\label{prop: Gorenstein_chi}
Let $A$ be a dg-algebra as in Setup~\ref{setup}. If $A$ is Gorenstein (Definition~\ref{def:gorenstein}), then $A$ satisfies condition $\chi$. 
\end{prop}
\begin{proof}
Let $M \in \Db(A)$. Since $A$ is Gorenstein, the natural map
$$
\RHom_A(\k,M) \to \RHom_{A^{\op}}(\RHom_A(M,A),
\RHom_A(\k,A))
$$
is a $\kk$-linear quasi-isomorphism. Using again that $A$ is Gorenstein, choose $a, n \in \Z$ such that there is an isomorphism $\RHom_A(\k, A) \simeq \k(a)[-n]$ in $\D(A)$. In fact, $\RHom_A(\k, A)$ and $\k(a)[-n]$ are isomorphic in $\D(A\otimes_{\k} A^{\op})$, since the quasi-isomorphisms 
$$
\RHom_A(\k, A) \xra{\simeq} \sigma^{\ge n} \RHom_A(\k, A) \xleftarrow{\simeq} \sigma^{\le n} \sigma^{\ge n} \RHom_A(\k, A) \cong \kk[-n](a)
$$
are morphisms of dg-$A$-$A$-bimodules; here $\sigma^{\ge n}$ and $\sigma^{\le n}$ denote smart truncations, as in \cite[Remark 2.7]{brown2023orlovs}. We therefore have a $\k$-linear quasi-isomorphism 
$$
  \RHom_{A^{\op}}(\RHom_A(M,A),
\RHom_A(\k,A)) \simeq \RHom_{A^{\op}}(\RHom_A(M,A), \k(a)[-n]).
$$
Since $A$ is Gorenstein, we have $\RHom_A(M,A)\in \Db(A^{\op})$. 
Choose a semifree resolution $F$ of the dg-$A^{\op}$-module $\RHom_A(M,A)$ as in \cite[Proposition 2.16]{brown2023orlovs}, so that
$$
\RHom_{A^{\op}}(\RHom_A(M,A), \k(a)[-n]) \cong \uHom_{A^{\op}}(F, \k(a)[-n]). 
$$
Finally, we have $F_j = 0$ for $j \ll 0$, which implies that $\uHom_{A^{\op}}(F, \k(a)[-n])_j = 0$ for $j \gg 0$. 
\end{proof}

Before we prove Theorem~\ref{thm:cd}, we establish the following technical result:

\begin{lem}\label{inf_internal_hom}
    Let $A$ be a dg-algebra as in Setup~\ref{setup}, $M \in \Db(A)$, and $N \in \D(A)$ such that $\inf(N) >-\infty$. We have $\inf(\RHom_A(M,N))\geq \inf(N)-\sup(M)$.
\end{lem}
\begin{proof}
Write $\ell \ce \inf(N)$, and let $\sigma^{\ge \ell} N$ denote the smart truncation of $N$ in cohomological degrees at least~$\ell$. There is a quasi-isomorphism $N \xra{\simeq} \sigma^{\ge \ell} N$ of dg-$A$-modules~\cite[Remark 2.7]{brown2023orlovs}. Let $F \xra{\simeq} M$ be a semifree resolution as in~\cite[Proposition 2.16]{brown2023orlovs}; it follows from the proof of this result that $F^i = 0$ for $i > \sup(M)$. We have $\RHom_A(M, N) \cong \uHom_A(F, \sigma^{\ge \ell} N)$, and the latter complex vanishes in cohomological degrees smaller than $\inf(N) - \sup(M)$. 
\end{proof}

\begin{proof}[Proof of Theorem~\ref{thm:cd}]
Let $j \in \Z$. Since $H^j(M)$ is finitely generated over $H^0(A)$, the triangle~\eqref{eqn:adjtri} implies that, to prove (1), it suffices to show that $H^j_\m(M)_{\ge i}$ is finitely generated over $H^0(A)$. We begin by showing that $\dim_\kk H^j_\m(M)_\ell < \infty$ for all $\ell \in \Z$. Remark~\ref{rem:easy} implies that $\dim_\k H^t(M)_\ell < \infty$ for all $\ell, t \in \Z$; since $R$ is a dualizing dg-module, $\RHom(M_{\ge 1}, R) \in \Db(A)$, and so Proposition~\ref{prop:ext} yields that $\dim_\k \RR^t\Gamma(\widetilde{M(\ell)}) < \infty$ for all $\ell, t \in \Z$ as well. The triangle~\eqref{eqn:adjtri} thus implies that $\dim_\kk H^j_\m(M)_\ell < \infty$ for all $\ell \in \Z$. To prove (1), we therefore need only show $H^j_\m(M)_\ell = 0$ for $\ell \gg 0$. 

Let $d \in \Z$. The short exact sequence $0 \to A_{\ge d} / A_{\ge d+1} \to A / A_{\ge d+1} \to A / A_{\ge d} \to 0$ induces the following exact sequence:
$$
\lext^{j-1}_{A}(A_{\ge d} / A_{\ge d+1},M)\xrightarrow{} \lext^{j}_{A}(A/A_{\ge d},M)\xrightarrow{} \lext^{j}_{A}(A/A_{\ge d+1},M) \xrightarrow{} \lext^{j}_{A}(A_{\ge d} / A_{\ge d+1},M).
$$
The dg-$A$-module $A_{\ge d} / A_{\ge d+1}$ is quasi-isomorphic to a finite direct sum of cohomological shifts of $\k(-d)$. 
Since $A$ satisfies condition~$\chi$ (Proposition~\ref{prop: Gorenstein_chi}), we conclude that 
there exists $d \gg 0$ such that 
$$
\lext^{j-1}_{A}(A_{\ge m} / A_{\ge m+1},M)_{\ge i} = \lext^{j}_{A}(A_{\ge m} / A_{\ge m+1},M)_{\ge i}= 0 \quad \text{for all $m \ge d$},
$$
which implies that the canonical map
$
\lext^{j}_{A}(A/A_{\ge m},M)_{\ge i} \xrightarrow{} \lext^{j}_{A}(A/A_{\ge m+1},M)_{\ge i}
$
is an isomorphism for all $m \ge d$. Thus, $H^j_{\m}(M)_{\ge i} = \lext^{j}_{A}(A/A_{\ge d},M)_{\ge i}$, and so, to prove (1), it suffices to show that $\lext^{j}_{A}(A/A_{\ge d},M)_{\ell} =~0$ for $\ell \gg 0$. Since $A/A_{\ge d} \in \D^{\tors}_{\gr}(A)$, and $\D^{\tors}_{\gr}(A)$ is generated by $\k(t)$ for $t \in \Z$ \cite[Lemma 3.1]{brown2023orlovs}, this holds since $A$ satisfies condition $\chi$.

   We now prove (2). Recall that $\RR^j\Gamma(\widetilde{M}) = \underset{d \to \infty}{\on{colim}}  \Ext^j_A(A_{\ge d}, M)$. By Lemma~\ref{inf_internal_hom} and the exactness of filtered colimits, it follows that $\RR^j\Gamma(\widetilde{M}) = 0$ for $j < \inf(M)$. On the other hand, Theorem~\ref{thm: Serre_duality} implies that $\RR^j\Gamma(\widetilde{M}) = \uHom_\k( \underset{d \to \infty}{\on{colim}} \Ext^{-j-1}_A(M_{\ge d}, R), \k)$. Part (2) therefore follows once again from Lemma~\ref{inf_internal_hom} and the exactness of filtered colimits. Part (3) is immediate from (1) and (2). 
\end{proof}

As an application of Theorem~\ref{thm:cd}, we obtain an extension of \cite[Theorem 1.3]{brown2023orlovs}, which is a generalization to the setting of dg-algebras of a theorem of Orlov \cite[Theorem 2.5]{Orlov2009}. More precisely, we deduce:

\begin{thm}
\label{thm:main}
Let $\kk$ be a field and $A$ a dg-algebra as in Setup~\ref{setup}. Suppose $A$ is Gorenstein (Definition~\ref{def:gorenstein}) with Gorenstein parameter $a$, and assume $A$ admits a balanced dualizing dg-module (\Cref{def:balanced_dualizing}). Write $\Dsing(A) \ce \Db(A) / \Perf(A)$, where $\Perf(A)$ denotes the thick subcategory of $\Db(A)$ generated by $A(i)$ for all $i \in \Z$. Let $q : \Db(A) \to \Dsing(A)$ and $\pi : \Db(A) \to \Dqgr(A)$ denote the canonical functors. The objects $\pi A(j) \in \Dqgr(A)$ and $q \k(j) \in \Dsing(A)$ are exceptional for all $j \in \Z$, and we have:
\begin{enumerate}
\item If $a > 0$, then for each $i \in \Z$, there is a fully faithful functor $\Phi_i : \Dsing(A) \to \Dqgr(A)$ and a semiorthogonal decomposition
$
\Dqgr(A) = \langle \pi A(-i - a + 1), \dots, \pi A(-i), \Phi_i\Dsing(A) \rangle.
$
\item If $a < 0$, then for each $i \in \Z$, there is a fully faithful functor $\Psi_i : \Dqgr(A) \to  \Dsing(A)$ and a semiorthogonal decomposition
$
\Dsing(A) = \langle q \kk(-i), \dots, q \kk (-i+a+1), \Psi_i \Dqgr(A)\rangle.
$
\item If $a = 0$, then there is an equivalence $\Dsing(A) \xra{\simeq} \Dqgr(A)$.
\end{enumerate}

\end{thm}

\begin{proof}
    Theorem~\ref{thm:cd}(3) provides an analogue of \cite[Proposition 3.6]{brown2023orlovs} in our setting; with this in hand, the proof of \Cref{thm:main} is exactly the same as that of \cite[Theorem 1.3]{brown2023orlovs}.
\end{proof}

\begin{remark}
A positive answer to Conjecture~\ref{conj:balanced} would allow one to remove from Theorems~\ref{thm:cd} and~\ref{thm:main} the assumption of the existence of a balanced dualizing dg-module.
\end{remark}

\bibliographystyle{amsalpha}
\bibliography{references}
\Addresses
\end{document}